\documentclass{article}
\usepackage{amsmath}
\usepackage{amsfonts}
\usepackage{amsthm}
\usepackage{amssymb}
\usepackage{color}
\usepackage{graphicx}
\usepackage{float}
\usepackage{subfigure}
\usepackage[font=small]{caption}
\usepackage{tikz}

\newtheorem{theorem}{Theorem}[section]
\newtheorem{lemma}[theorem]{Lemma}

\newtheorem{example}[theorem]{Example}
\newtheorem{corollary}[theorem]{Corollary}

\newtheorem{claim}[theorem]{Claim}

\theoremstyle{definition}
\newtheorem{definition}[theorem]{Definition}
\theoremstyle{remark}
\newtheorem{remark}[theorem]{Remark}

\newcommand\remove[1]{}

\def\f2{\mathbb{F}_2}

\newcommand{\1}{\mathbf{1}}

\newcommand{\diam}{{\rm diam}\hskip0.02cm}

\begin{document}

\title{\LARGE Nonexistence of embeddings with uniformly bounded
distortions of Laakso graphs into diamond graphs}

\author{Sofiya~Ostrovska\\
\\
Department of Mathematics\\ Atilim University \\ 06836 Incek,
Ankara, TURKEY \\
\textit{E-mail address}:
\texttt{sofia.ostrovska@atilim.edu.tr}\\
\\
and \\
\\
Mikhail~I.~Ostrovskii\\
\\
Department of Mathematics and Computer Science\\
St. John's University\\
8000 Utopia Parkway\\
Queens, NY 11439, USA \\
\textit{E-mail address}: \texttt{ostrovsm@stjohns.edu}\\
Phone: 718-990-2469\\
Fax: 718-990-1650}

\date{~}
\maketitle

\begin{abstract}

Diamond graphs and Laakso graphs are important examples in the
theory of metric embeddings. Many results for these families of
graphs are similar to each other. In this connection, it is
natural to ask whether one of these families admits uniformly
bilipschitz embeddings into the other. The well-known fact that
Laakso graphs are uniformly doubling but diamond graphs are not,
immediately implies that diamond graphs do not admit uniformly
bilipschitz embeddings into Laakso graphs. The main goal of this
paper is to prove that Laakso graphs do not admit uniformly
bilipschitz embeddings into diamond graphs.

\end{abstract}

{\small \noindent{\bf Keywords.} diamond graphs, doubling metric
space, Laakso space, Lipschitz map}\medskip

{\small \noindent{\bf 2010 Mathematics Subject Classification.}
Primary: 46B85; Secondary: 05C12, 30L05.}

\begin{large}

\section{Introduction}

Diamond graphs and Laakso graphs are important examples in the
theory of metric embeddings, see \cite{BC05,Laa00,LP01,LN04,
MN13,NR03,Ost11,Ost13,Ost14a} . Many results for these families of
graphs are similar to each other, see the example after
Definitions \ref{D:Diamonds} and \ref{D:Laakso} and in Section
\ref{S:Conseq}. In this connection, the question emerges: does one
of these families admit uniformly bilipschitz embeddings into the
other. The well-known fact that Laakso graphs are uniformly
doubling but diamond graphs are not uniformly doubling - see
Definition \ref{D:doubling} and the subsequent discussion  -
immediately implies that diamond graphs do not admit uniformly
bilipschitz embeddings into Laakso graphs. The main goal of this
paper is to prove that Laakso graphs do not admit uniformly
bilipschitz embeddings into diamond graphs.

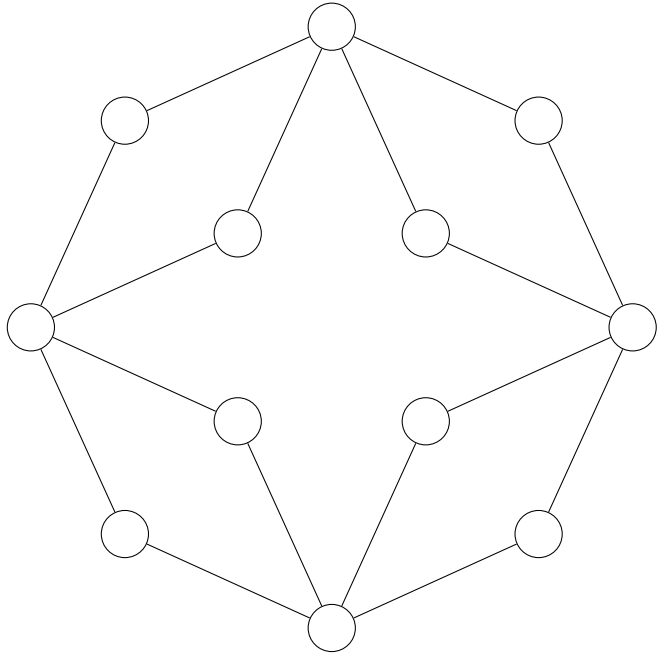
\begin{figure}
\begin{center}
{
\begin{tikzpicture}
  [scale=.25,auto=left,every node/.style={circle,draw}]
  \node (n1) at (16,0) {\hbox{~~~}};
  \node (n2) at (5,5)  {\hbox{~~~}};
  \node (n3) at (11,11)  {\hbox{~~~}};
  \node (n4) at (0,16) {\hbox{~~~}};
  \node (n5) at (5,27)  {\hbox{~~~}};
  \node (n6) at (11,21)  {\hbox{~~~}};
  \node (n7) at (16,32) {\hbox{~~~}};
  \node (n8) at (21,21)  {\hbox{~~~}};
  \node (n9) at (27,27)  {\hbox{~~~}};
  \node (n10) at (32,16) {\hbox{~~~}};
  \node (n11) at (21,11)  {\hbox{~~~}};
  \node (n12) at (27,5)  {\hbox{~~~}};

  \foreach \from/\to in {n1/n2,n1/n3,n2/n4,n3/n4,n4/n5,n4/n6,n6/n7,n5/n7,n7/n8,n7/n9,n8/n10,n9/n10,n10/n11,n10/n12,n11/n1,n12/n1}
    \draw (\from) -- (\to);

\end{tikzpicture}
} \caption{Diamond $D_2$.}\label{F:Diamond2}
\end{center}
\end{figure}

To the best of our knowledge, the first paper in which diamond
graphs $\{D_n\}_{n=0}^\infty$ were used in Metric Geometry is the
conference version of \cite{GNRS04}, which was published in 1999.

\begin{definition}\label{D:Diamonds}
Diamond graphs $\{D_n\}_{n=0}^\infty$ are defined recursively: The
{\it diamond graph} of level $0$ has two vertices joined by an
edge of length $1$ and is denoted by $D_0$. The {\it diamond
graph} $D_n$ is obtained from $D_{n-1}$ in the following way.
Given an edge $uv\in E(D_{n-1})$, it is replaced by a
quadrilateral $u, a, v, b$, with edges $ua$, $av$, $vb$, $bu$.
(See Figure \ref{F:Diamond2}.)
\medskip

Two different normalizations of the graphs $\{D_n\}_{n=1}^\infty$
can be found in the literature:

\begin{itemize}

\item {\it Unweighted diamonds:} Each edge has length $1$.

\item {\it Weighted diamonds:} Each edge of $D_n$ has length
$2^{-n}$.
\end{itemize}

In both cases, we endow the vertex sets of $\{D_n\}_{n=0}^\infty$
with their shortest path metrics.

For weighted diamonds,  the identity map $D_{n-1}\mapsto D_n$ is
an isometry and, in this case, the union of $D_n$  endowed with
the metric induced from $\{D_n\}_{n=0}^\infty$ is called the {\it
infinite diamond} and denoted by $D_\omega$.
\end{definition}

Another family of graphs considered in the present article is that
of Laakso graphs. The Laakso graphs were introduced in
\cite{LP01}, but they were inspired by the construction of Laakso
\cite{Laa00}.

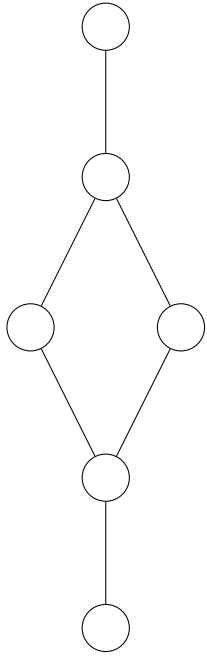
\begin{figure}
\begin{center}
{
\begin{tikzpicture}
  [scale=.25,auto=left,every node/.style={circle,draw}]
  \node (n1) at (16,0) {\hbox{~~~}};
  \node (n2) at (16,8)  {\hbox{~~~}};
  \node (n3) at (12,16) {\hbox{~~~}};
  \node (n4) at (20,16)  {\hbox{~~~}};
  \node (n5) at (16,24)  {\hbox{~~~}};
  \node (n6) at (16,32) {\hbox{~~~}};

\foreach \from/\to in {n1/n2,n2/n3,n2/n4,n4/n5,n3/n5,n5/n6}
    \draw (\from) -- (\to);

\end{tikzpicture}
} \caption{Laakso graph $L_1$.}\label{F:Laakso}
\end{center}
\end{figure}

\begin{definition}\label{D:Laakso}
Laakso graphs $\{L_n\}_{n=0}^\infty$ are defined recursively: The
{\it Laakso graph} of level $0$ has two vertices joined by an edge
of length $1$ and is denoted $L_0$. The {\it Laakso graph} $L_n$
is obtained from $L_{n-1}$ according to the following procedure.
Each edge $uv\in E(L_{n-1})$ is replaced by the graph $L_1$
exhibited in Figure \ref{F:Laakso}, the vertices $u$ and $v$ are
identified with the vertices of degree $1$ of $L_1$.
\medskip

Similarly to the case of diamond graphs, the two different
normalizations of the graphs $\{L_n\}_{n=1}^\infty$ are used:

\begin{itemize}

\item {\it Unweighted Laakso graphs:} Each edge has length $1$.

\item {\it Weighted Laakso graphs:} Each edge of $L_n$ has length
$4^{-n}$.
\end{itemize}

In both situations,  we endow vertex sets of
$\{L_n\}_{n=0}^\infty$ with their shortest path metrics. In the
case of weighted Laakso graphs, the identity map $L_{n-1}\mapsto
L_n$ is an isometry and the union of $L_n$, endowed with the
metric induced from $\{L_n\}_{n=0}^\infty$, is called the {\it
Laakso space} and denoted by $L_\omega$.
\end{definition}

Many known results for one of the aforementioned  families admit
analogues for the other. For example, it is known \cite{JS09} that
both of the families can be used to characterize superreflexivity.
Further, the two families consist of planar graphs with poor
embeddability into Hilbert space, see \cite{Laa00,LP01,NR03}.
Moreover, in many situations, the proofs used for one of the
families can be easily adjusted to work for the other family. For
instance, this is the case for the Markov convexity (see
\cite[Section 3]{MN13} and Section \ref{S:Conseq} below). There
are similarities between results for the infinite diamond and the
Laakso space, too. For example, neither of spaces $D_\omega$ and
$L_\omega$ admits bilipschitz embeddings into any Banach space
with the Radon-Nikod\'ym property \cite{CK09,Ost11}.

On the other hand, the families $\{D_n\}$ and $\{L_n\}$ are not
alike in some important metrical respects, and the corresponding
properties of the Laakso graphs were among the reasons for their
introduction. To exemplify the differences, it can be mentioned
that the Laakso graphs are uniformly doubling (see \cite[Theorem
2.3]{LP01}) and unweighted Laakso graphs have uniformly bounded
geometry, whereas diamond graphs do not possess any of these
properties. These facts are well known. Nevertheless, for the
convenience of the readers, they will be proved after recalling
the necessary definitions as a suitable reference is not
available.

\begin{definition}\label{D:doubling} {\rm (i)} A metric space $X$ is called {\it doubling} if there is a constant $D\ge 1$ such that every bounded set $B$ in $X$ can be
covered by at most
$D$ sets of diameter $\diam(B)/2$. A sequence of metric spaces is
called {\it uniformly doubling} if all of them are doubling and
the constant $D$ can be chosen to be the same for all of them.

{\rm (ii)} A metric space $X$ is said to have {\it bounded
geometry} if there is a function $M:(0,\infty)\to(0,\infty)$ such
that each ball of radius $r$ in $X$ has at most $M(r)$ elements. A
sequence of metric spaces is said to have {\it uniformly bounded
geometry} if the function $M(r)$ can be chosen to be the same for
all of them.
\end{definition}

It is easy to verify that a sequence of unweighted graphs has
uniformly bounded geometry if and only if they have uniformly
bounded degrees. This observation implies immediately that
unweighted $\{L_n\}$ has uniformly bounded geometry, but $\{D_n\}$
does not.\medskip

The fact that $\{D_n\}$ is not uniformly doubling - it is easy to
see that this statement does not depend on normalization - can be
shown in the following way. Consider unweighted diamonds and call
one of the vertices of $D_0$ the {\it top} and the other the {\it
bottom}. Define the {\it top} and the {\it bottom} of $D_n$ as
vertices which evolved from the top and the bottom of $D_0$,
respectively. Now, consider the ball of radius $1$ centered at the
bottom of $D_n$. It contains $2^n+1$ elements (including the
bottom) and has diameter $2$. On the other hand, it can be readily
seen that any subset of this ball of diameter $1$ contains two
elements. The fact that $\{D_n\}$ is not uniformly doubling
follows. Many interesting results on doubling metric spaces can be
found in \cite{Hei01}.
\medskip

Our goal is to show that the diamond and Laakso graphs are not
alike in one more respect in terms of the next standard
definition:

\begin{definition}\label{D:LipBilip}
{\rm (i)} Let $0\le C<\infty$. A map $f: (A,d_A)\to (Y,d_Y)$
between two metric spaces is called $C$-{\it Lipschitz} if
\[\forall u,v\in A\quad d_Y(f(u),f(v))\le Cd_A(u,v).\] A map $f$
is called {\it Lipschitz} if it is $C$-Lipschitz for some $0\le
C<\infty$.\medskip

{\rm (ii)} Let $1\le C<\infty$. A map $f:A\to Y$ is called a {\it
$C$-bilipschitz embedding} if there exists $r>0$ such that
\begin{equation}\label{E:MapDist}\forall u,v\in A\quad rd_A(u,v)\le
d_Y(f(u),f(v))\le rCd_A(u,v).\end{equation} A {\it bilipschitz
embedding} is an embedding which is $C$-bilipschitz for some $1\le
C<\infty$. The smallest constant $C$ for which there exist $r>0$
such that \eqref{E:MapDist} is satisfied is called the {\it
distortion} of $f$.

{\rm (iii)} Let $\{M_n\}_{n=1}^\infty$ and $\{R_n\}_{n=1}^\infty$
be two sequences of metric spaces. We say that
$\{M_n\}_{n=1}^\infty$ admits {\it uniformly bilipschitz
embeddings} into $\{R_n\}_{n=1}^\infty$ if for each
$n\in\mathbb{N}$ there is $m(n)\in\mathbb{N}$ and a bilipschitz
map $f_n:M_n\to R_{m(n)}$ such that the distortions of $\{f_n\}$
are uniformly bounded.
\end{definition}

\begin{remark} It is not difficult to see that a sequence of metric spaces which admits uniformly bilipschitz embeddings into a sequence of uniformly
doubling spaces is itself uniformly doubling. Together with the
mentioned above fact that the graphs $\{L_n\}_{n=0}^\infty$ are
uniformly doubling while $\{D_n\}_{n=0}^\infty$ are not, this
observation implies that $\{D_n\}_{n=0}^\infty$ does not admit
uniformly bilipschitz embeddings into $\{L_n\}_{n=0}^\infty$.
\end{remark}

The  present  paper aims to prove the non-embeddability in the
opposite  direction. This has been achieved by the next theorem,
which constitutes the main result of this work.

\begin{theorem}\label{T:LaaksoIntoDmnds} The sequence of Laakso graphs does not admit uniformly bilipschitz
embeddings into the sequence of diamond graphs.
\end{theorem}

We refer to \cite{BM08} for graph-theoretical terminology and to
\cite{Ost13} for terminology of the theory of metric embeddings.

\section{Proof of the main result}

Since,  clearly, the validity of Theorem \ref{T:LaaksoIntoDmnds}
does not depend on whether weighted or unweighted version  of
$\{D_n\}$ and $\{L_n\}$  is considered, our attention in this
section is restricted to unweighted versions of $\{D_n\}$ and
$\{L_n\}$. The vertex set of a graph $G$ will be denoted by
$V(G),$ and the edge set by $E(G).$
\medskip

To prove Theorem \ref{T:LaaksoIntoDmnds}, it suffices to show
that, for any $k\in\mathbb{N}$, no matter how the numbers
$m(n)\in\mathbb{N}$ and $p(n)\in\mathbb{N}\cup\{0\}$ are chosen,
it is impossible to find maps $F_n: V(L_n)\to V(D_{m(n)})$ such
that
\begin{equation}\label{E:bilip}\forall n~\forall u,v\in
V(L_n)\quad 4^{p(n)}d_{L_n}(u,v)\le d_{D_{m(n)}}(F_nu,F_nv)\le
4^k\cdot 4^{p(n)}d_{L_n}(u,v).\end{equation}

It should be pointed out that there is no need to consider
negative values of $p(n)$. Indeed, in such cases, one may replace
$D_{m(n)}$ by $D_{m(n)-2p(n)}$ and use the natural map of
$D_{m(n)}$ into $D_{m(n)-2p(n)}$, which multiplies all distances
by $4^{-p(n)}$.\medskip

To show the impossibility of finding $\{F_n\}$ satisfying the
condition above assume that such maps exist. Observe that $L_n$
$(n\ge 1)$ contains $4^h$-cycles isometrically for all
$h\in\{1,\dots, n\}$. Some families of such cycles, labelled by
quaternary trees $Q_a$ with $a\in\mathbb{N}$ generations, will be
used in the forthcoming reasonings.\medskip

\begin{definition}\label{D:quart}
The {\it quaternary tree} $Q_a$ of {\it depth} $a$ is defined as
the graph whose vertex set $V(Q_a)$ is the set of all sequences of
length $\le a$ with entries $\{0,1,2,3\}$, including the empty
sequence $\emptyset$; and the edge-set $E(Q_a)$ is defined by the
following rule: two sequences are joined by an edge if and only if
one of them is obtained from the other by adding an element at the
right end.
\end{definition}

For any $s,t\in\{1,\dots,n\}$, $s>t$, we can find a set
$\mathcal{C}_{s,t}$ of cycles in $L_n$ labelled by $Q_{s-t}$ and
satisfying the next two conditions:

\begin{itemize}

\item The cycle $c_\emptyset$ corresponding to $\emptyset$ has
length $4^{s}$ and common vertices with the cycle of length $4^n$
in $L_n$ (it coincides with the latter cycle if $s=n$). Note that
unless $s=n$ there are many such cycles, we pick one of them.

\item Let $\tau$ be a $\{0,1,2,3\}$-sequence with $m$ entries
(where $m<s-t$), for which we have already defined the
corresponding cycle $c_\tau$ of length $4^{s-m}$. Then cycles
$c_{\tau,0}$, $c_{\tau,1}$, $c_{\tau,2}$, and $c_{\tau,3}$
corresponding to sequences $\{\tau,0\}$, $\{\tau,1\}$,
$\{\tau,2\}$, $\{\tau,3\}\in Q_{s-t}$ are cycles of length
$4^{s-m-1}$ such that no pair of them has a common vertex, but
each of them has common vertices with $c_\tau$.
\end{itemize}

Here, $L_n$ and $D_n$ are regarded not only as combinatorial
graphs, but also as $1$-dimensio\-nal simplicial complexes
according to the following  standard procedure: each edge is
identified with a line segment of length $1$ and the distance is
the shortest path distance (see \cite[pp.~82--83]{RR98}). There is
a natural way to extend the maps $F_n$ to these simplicial
complexes in such a manner that the Lipschitz constant of $F_n$
does not change: pick, for each edge $uv\in E(L_n)$, one of the
shortest paths between $F_nu$ and $F_nv$ and map the edge into the
path so that, for any $t\in uv$, one has:
\[d_{D_{m(n)}}(F_nu,F_nt)=d_{L_n}(u,t)\cdot\frac{d_{D_{m(n)}}(F_nu,F_nv)}{d_{L_n}(u,v)}.
\]

Select a cycle of length $4^h$ in $L_n$ and denote it $C_{4^h}$.
The assumption on $F_n$ implies that the images of consecutive
vertices of $C_{4^h}$ under $F_n$ can be joined by the shortest
paths of lengths between $4^{p(n)}$ and $4^{p(n)+k}$. As a result,
one obtains a closed walk of length between $4^{p(n)+h}$ and
$4^{p(n)+k+h}$. It can be shown by example that this closed walk
does not have to be a cycle as it can have additional
self-intersections. However, as it will be proved in Lemma
\ref{L:LongCycle}, for $h$ which is substantially larger than $k$,
- in the sense described below - the image $F_nC_{4^h}$ has to
contain a large cycle in $D_{m(n)}$.\medskip

The following definitions are used in the sequel. Some of them
were introduced in \cite{JS09} and \cite{Ost14a}.

\begin{definition}\label{D:SubdiamPrinC} (i) A {\it subdiamond} of $D_{n}$ is a
subgraph which evolved from an edge of some $D_k$ for $0\le k\le
n$. Similar to  the notions of the {\it top} and {\it bottom} of
$D_{n}$ introduced above, we define the {\it top (bottom) of a
subdiamond} $S$ as its vertex which is the closest to the top
(bottom) of $D_{n}$ among vertices of $S$.

(ii) When a subdiamond $S$ evolves from an edge $uv$, in the first
step we obtain a quadrangle $u,a,v,b$. The vertex of $S$
corresponding to $a$ is called the {\it leftmost}, while  the
vertex of $S$ corresponding to $b$ is called the {\it rightmost}
vertex of $S$. Obviously, a choice of left and right here is
arbitrary, but sketching a diamond in the plane - recall that all
of $\{D_n\}$ and $\{L_n\}$ are planar graphs - one can follow our
choices. The distance between the top and the bottom of a
subdiamond is called the {\it height} of the subdiamond.

(iii) By a {\it principal cycle} in a subdiamond $S$ of $D_{n}$ we
mean a cycle which consists of two parts: a path from top to
bottom which passes through the leftmost vertex and a path from
bottom to top which passes through the rightmost vertex. A
principal cycle is considered not only as a graph-theoretical
cycle, but also as a $1$-dimensional simplicial complex
homeomorphic to a circle.
\end{definition}

The next lemma reveals a simple property of cycles in diamonds.

\begin{lemma}\label{L:CyclesInDiam} Each cycle $C$ in $D_{n}$
is a principal cycle in one of subdiamonds of $D_{n}$. Each
principal cycle in a subdiamond $S$ of height $2^t$ has length
$2^{t+1}$.
\end{lemma}

\begin{proof} Let $S$ be the smallest subdiamond in $D_n$
containing the cycle $C$. Let $2^t$ be the height of $S$ and let
$S_1,S_2,S_3,$ and $S_4$ be the subdiamonds of $S$ of height
$2^{t-1}$. See Figure \ref{F:Diamond2}. Since $C$ is not contained
in any of $S_1,S_2,S_3,S_4$, it has common edges with at least two
of them, say, $S_1$ and $S_2$. It is clear that either the pair
top-bottom of $S$ or the pair leftmost-rightmost vertices
separates $S_1$ from $S_2$ in $C$. Therefore, both vertices of the
corresponding pair should be in the cycle. In the top-bottom case,
there should be two paths joining top and bottom. Since the
leftmost-rightmost vertices form a cut, one of the paths should go
through the leftmost vertex, and the other through the rightmost
vertex. To show that each of them is of length $2^t$, we use the
induction on $t$.
\medskip

For $t=1$ the situation is clear (sketch $D_1$). Assume that  the
statement holds for $t=k$ and consider the case $t=k+1$. Then the
path goes either through the leftmost or the rightmost vertex. In
any event, it goes from top to bottom of each of the subdiamonds
on the corresponding side. Using the induction hypothesis, one
obtains the desired result.

The leftmost-rightmost case is treated likewise.
\end{proof}

\begin{lemma}\label{L:LongCycle} The subgraph $M$ in $D_{m(n)}$ spanned by the closed
walk $F_nC_{4^h}$ contains a cycle of length $\ell$ satisfying
\begin{equation}\label{E:ell}\ell\ge 4^{p(n)}\left(\frac{4^h}3-1-4^k\right).\end{equation}
\end{lemma}

\begin{proof} Assume the contrary, that is, let the longest cycle contained
in $F_nC_{4^h}$ be of length
$<4^{p(n)}\left(\frac{4^h}3-1-4^k\right)$

Denote by $\nu$  the largest integer satisfying
\begin{equation}\label{E:nu}2^{\nu}< 4^{p(n)}\left(\frac{4^h}3-1-4^k\right).\end{equation}

We collapse  all subdiamonds which have principal cycles contained
in $F_nC_{4^h}$. By this we mean that, for each subdiamond $S$ of
height $2^\mu$ with $\mu\le\nu$, such that one of its principal
cycles is contained in $F_nC_{4^h}$, but $S$ is not (properly)
contained in a larger subdiamond whose principal cycle is also
contained in $F_nC_{4^h}$, we do the following. Replace $S$ by a
path of length $2^\mu$ and map all edges of the subdiamond onto
edges of the path in such a way that their distances from the top
and bottom of the subdiamond are preserved. Denote by $G$ the
graph obtained by applying this procedure to $D_{m(n)}$, and by
$P$ the corresponding map, viewed both as a map $V(D_{m(n)})\to
V(G)$ and $E(D_{m(n)})\to E(G)$.\medskip

Let us prove that $PF_nC_{4^h}$ is a tree. Indeed, it is obvious
that $PF_nC_{4^h}$ cannot contain a cycle $C$ of the form $P\tilde
C$, where $\tilde C$ is a cycle in $F_nC_{4^h}$ since all such
cycles $\tilde C$ were collapsed. Therefore, any of the preimages
$\tilde C$ of $C$ in $F_nC_{4^h}$ cannot be a cycle. This implies
that $C$ contains some of the paths obtained as collapsed
subdiamonds, and that, in at least one of them, the preimage of
this path in $F_nC_{4^h}$ cannot be connected. However, this is
impossible according to the preceding definitions.
\medskip

Therefore, by virtue  of \cite[Proposition 5.1]{RR98}, there are
two points $x,y$ in the cycle $C_{4^h}$ (regarded as a
$1$-dimensional simplicial complex) at distance $d_{L_n}(x,y)\ge
4^{h}/3$, such that $PF_nx=PF_ny$. Since $2^\nu$ is the diameter
of the largest subdiamond which was collapsed and two distinct
points in different collapsed subdiamonds are not mapped to the
same point, the distance between two preimages of the same point
under $P$ is at most $2^\nu$. Thence, $d_{D_{m(n)}}(F_nx,F_ny)\le
2^\nu$. As points $x$ and $y$ do not have to be vertices, we find
the closest vertices of $C_{4^h}$ to them, say, $v_x$ and $v_y$.
Then $d_{L_n}(x,v_x)\le\frac12$ and $d_{L_n}(y,v_y)\le\frac12$,
whence
\[d_{D_{m(n)}}(F_nv_x,F_nv_y)\le 2^\nu+4^{p(n)+k}.\]

On the other hand, inequality $d_{L_n}(x,y)\ge 4^{h}/3$ yields:
\[d_{D_{m(n)}}(F_nv_x,F_nv_y)\ge
\left(\frac{4^{h}}3-1\right)4^{p(n)}.\] Thus, $2^\nu\ge
4^{p(n)}\left(\frac{4^h}3-1- 4^k\right)$, contrary to
\eqref{E:nu}.
\end{proof}

This lemma implies the following claim.

\begin{claim}\label{C:CinC} If $h$ satisfies
\begin{equation}\label{E:ellLarge}\left(\frac{4^h}3-1-4^k\right)\ge 4^{h-1},\end{equation}
then the $F_n$-image of any (isometric) cycle $C_{4^h}$ of length
$4^h$ in $L_n$ contains a cycle whose  length in $D_{m(n)}$ is
between $4^{p(n)+h-1}$ and $4^{p(n)+h+k}$.
\end{claim}

Notice that the uniqueness of such a cycle in $D_{m(n)}$ is not
asserted. Denote by $A(C_{4^h})$ any of the cycles satisfying the
conditions of Claim \ref{C:CinC} and by $S(C_{4^h})$ the
subdiamond for which $A(C_{4^h})$ is a principal cycle.\medskip

Now,  consider the collection $\mathcal{C}_{s,t}$ of cycles in
$L_n$ introduced after Definition \ref{D:quart} and labelled by
$Q_{s-t}$. We pick $n>s>t$ in such a way that $h=t$ satisfies
\eqref{E:ellLarge}, and, in addition, the three inequalities below
hold:
\begin{equation}\label{E:n-s}
n>s+k,
\end{equation}

\begin{equation}\label{E:s-t}
s-t>2(k+1),
\end{equation}

\begin{equation}\label{E:t}
t>10(k+1).
\end{equation}

At this stage, the  {\bf question} arises: Given a cycle
$A(c_\tau)$ with sequence $\tau$ of length $m<s-t$, what can be
said about the cycles $A(c_{\tau,0}),A(c_{\tau,1}),A(c_{\tau,2})$,
$A(c_{\tau,3})$ and the corresponding subdiamonds $S(c_{\tau,0})$,
$S(c_{\tau,1}),S(c_{\tau,2})$, and $S(c_{\tau,3})$? The writing
$\tau,0$ means the sequence obtained from $\tau$ by adding $0$ at
the right end.
\medskip

Lemmas \ref{L:TopBot}, \ref{L:DoNotTouch}, and \ref{L:Decr8} below
provide an answer to this question.\medskip

\begin{lemma}\label{L:TopBot} The subdiamonds $S(c_{\tau,0}),S(c_{\tau,1}),S(c_{\tau,2})$, and
$S(c_{\tau,3})$ are contained in $S(c_\tau)$, but they contain
neither the top nor the bottom of $S(c_\tau)$.
\end{lemma}

\begin{proof} Since  $\diam (L_n)=4^n$, we have $\diam(F_nL_n)\ge
4^{p(n)+n}$. On the other hand, the condition $\diam
(c_\tau)=\frac12\,4^{s-m}$ implies $\diam S(c_\tau)=\diam
A(c_\tau)\le\frac12 4^{s-m+p(n)+k}\le \frac12 4^{s+p(n)+k}$, where
the equality  reflects an easy property of principal cycles.
Meanwhile, inequality \eqref{E:n-s} implies that $F_nL_n$ cannot
be contained in $S(c_\tau)$.

Notice that the deletion of $\{v_t,v_b\}$, where $v_t$ is the top
and $v_b$ is the bottom of $S(c_\tau)$ splits the diamond
$D_{m(n)}$ into three pieces: the part containing the rightmost
vertex of $S(c_\tau)$, the part containing the leftmost vertex of
$S(c_\tau)$, and the complement of $S(c_\tau)$. Using the previous
paragraph and the definition of $S(c_\tau),$ one concludes that
each of these parts has nonempty preimages under $F_n^{-1}$.
Hence, preimages of $v_t$ and $v_b$ split $L_n$ into at least
three connected components: the component containing the preimage
of the leftmost vertex, the component containing the preimage of
the rightmost vertex, and the component containing the vertices of
$L_n$ with the maximal distance to $c_\tau$.

It can be observed that the subsets $F_n^{-1}(v_t)$ and
$F_n^{-1}(v_b)$ in $L_n$ have diameters $\le 4^k+1$, and the same
is also true for the preimage of any point of $F_nL_n$. Indeed,
assume the contrary, that is, let $F_nx=F_ny$ and
$d_{L_n}(x,y)>4^k+1$. Let $v_x,v_y\in V(L_n)$ be the nearest to
$x$ and $y$ vertices. Then $d_{L_n}(v_x,v_y)>4^k$, but
$d_{D_{m(n)}}(F_nv_x,F_nv_y)\le 4^{p(n)+k}$. This, however,
contradicts \eqref{E:bilip}.

To show that the preimages of $v_t$ and $v_b$ cannot intersect any
of $c_{\tau,0},c_{\tau,1},c_{\tau,2}$, and $c_{\tau,3}$, some more
information on the action of the map $F_n$ on $c_\tau$ is
required.\medskip

The cycle $c_\tau$ has length $4^{s-m}$ and $s-m>t$. By Claim
\ref{C:CinC}, the length of $A(c_\tau)$ is at least
$4^{p(n)+s-m-1}\ge 4^{p(n)+t}$.

We divide the cycle $A(c_\tau)$ into pieces of length
$4^{p(n)+k}$. By Lemma \ref{L:CyclesInDiam} the length of every
cycle in $D_{m(n)}$ is a power of $2$. Since by \eqref{E:t},
$4^{p(n)+t}>4^{p(n)+k}$, the length of $A(c_\tau)$ is divisible by
$4^{p(n)+k}$. We get at least $4^{s-m-k-1}$ pieces whose endpoints
we denote by $u_r$  and the preimage of each $u_r$ in $c_\tau$ by
$a_r$. In the case of several preimages, we pick one of them. We
may and shall assume that both $v_t$ and $v_b$ are among
$\{u_r\}$. Inequality \eqref{E:bilip} implies that
\begin{equation}\label{E:dista_r}
\forall q\in\{1,2,\dots,4^{p(n)+s-m-2}\},\quad
d_{L_n}(a_r,a_{r+q})\in [q,q\cdot 4^k].
\end{equation}

One can show by example that the sequence $\{a_r\}$ does not have
to be located on $c_\tau$ in a cyclic order. However, it will be
shown that if we replace $4^{p(n)+k}$ in the definition of
$\{u_r\}$ by $4^{p(n)+2k}$, the preimages $\{b_r\}$ of the
resulting sequence $\{v_r\}$ are located on $c_\tau$ in a cyclic
order.\medskip

By \eqref{E:t},  $t>10 k$, whence the number $4^k$ is
significantly smaller than $4^{s-m}$ (the length of $c_\tau$), and
the following statement is meaningful.

\begin{lemma}\label{L:SameSide} The points $a_{4^k}$ and $a_{4^k+1}$ are on the same
side of $a_0$. Furthermore, all points
$a_{4^k},a_{4^k+1},a_{4^k+2},\dots, a_q$ with $q=\frac\124^{9k}$
are on the same side of $a_0$.
\end{lemma}

\begin{proof} The assumption that $a_{4^k}$ and $a_{4^k+1}$
are on different sides of $a_0$ implies that
\[\begin{split}d_{L_n}(a_{4^k},a_{4^k+1})&\ge\min\{d_{L_n}(a_{4^k},a_0)+d_{L_n}(a_0,a_{4^k+1}),
4^{s-m}-(2\cdot 4^k+1)4^k\}
\\&\stackrel{\eqref{E:dista_r}}{\ge}\min\{2\cdot 4^k+1,
4^{s-m}-(2\cdot 4^k+1)4^k\}\stackrel{\eqref{E:t}}{=}2\cdot 4^k+1,
\end{split}\] while the distance between $u_{4^k}$ and $u_{4^k+1}$ is
$4^{p(n)+k}$. We obtain a contradiction.

The second statement can be proved by using this reasoning along
with  the induction.
\end{proof}

Now, let $v_r=u_{r\cdot 4^k}$, $r=1,2,\dots$, and let
$b_r=a_{r\cdot 4^k}$ be the preimages of $v_r$.

\begin{lemma}\label{L:NaturalOrder} The sequence $\{b_r\}$ is placed on the
cycle $c_\tau$ in its natural order.
\end{lemma}

\begin{proof} Assume the contrary, that is, for some $r$ the order is $b_r, b_{r+2},
b_{r+1}$. We apply Lemma \ref{L:SameSide} to two pairs: (1)
$b_r=a_{r\cdot 4^k}$ (playing the role of $a_0$) and
$b_{r+2}=a_{(r+2)\cdot 4^k}$ and (2) $b_{r+1}=a_{(r+1)\cdot 4^k}$
(playing the role of $a_0$) and $b_{r+2}=a_{(r+2)\cdot 4^k}$. As a
result we conclude, by induction, that points $a_{(r+2)4^k+i}$,
$i=1,2,\dots,4^{8k}$ should be between $b_r$ and $b_{r+2}$. By
\eqref{E:dista_r}, this is impossible.
\end{proof}

It is easy to see that one may assume that $v_t$ and $v_b$ are
among $\{v_r\}$, that is, the adopted notation is consistent. With
some abuse of notation, assume that $t$ and $b$ are not only the
abbreviations for  `top' and `bottom', but also integers. We
denote the corresponding elements of $\{b_r\}$ by $b_t$ and $b_b$.

Let us denote the vertices at which $c_\tau$ is connected to the
largest cycle of $L_n$ by $c_T$ and $c_B$, respectively.
Considering pieces of  $c_\tau$ between $b_{t-2}$ and $b_{t+2}$
and also between  $b_{b-2}$ and $b_{b+2},$ we claim that both
$c_T$ and $c_B$ belong to one of these pieces of $c_\tau.$ Indeed,
suppose that $c_T$ is not in any of the mentioned pieces. Since
$F_n^{-1}(v_t)$ and $F_n^{-1}(v_b)$ are contained in the pieces,
one may assume without loss of generality that after the removal
of $F_n^{-1}(v_t)$ and $F_n^{-1}(v_b)$, the vertex $c_T$ will end
up in the same component of $L_n$ as $b_{b-2}$. It leads to a
contradiction because it implies that $b_{b-2}$ and the part of
$L_n$ containing points which are the most distant from $c_\tau$
are in the same component. The proof for $c_B$ is similar. Note
that, using the triangle inequality it is easy to verify that
$c_T$ and $c_B$ are in different pieces.

Observe that the distance between $c_T$ or $c_B$ and any of
$c_{\tau,0},c_{\tau,1},c_{\tau,2}$, and $c_{\tau,3}$ is at least
$4^{s-m-2}\ge 4^{t-1}$. On the other hand, the distance between
$b_{b-2}$ and $b_{b+2}$ is $\le 4^{3k}$. Applying inequality
\eqref{E:t}, we derive the statement of the lemma.
\end{proof}

\begin{lemma}\label{L:DoNotTouch} The cycles
$A(c_{\tau,0}),A(c_{\tau,1}),A(c_{\tau,2})$, and $A(c_{\tau,3})$
are disjoint.
\end{lemma}

\begin{proof} We are going to show that the existence of  common
vertices contradicts the bilipschitz condition \eqref{E:bilip}. In
fact, if there are points $z\in c_{\tau, i}$, $w\in c_{\tau, j}$
$i,j\in\{0,1,2,3\}$, for which $i\ne j$ and their images coincide,
then there exist vertices $v_z\in c_{\tau, i}$ and $v_w \in
c_{\tau, j}$ in the cycles with $d_{L_n}(z,v_z)\le\frac12$ and
$d_{L_n}(w,v_w)\le\frac12$. Therefore,
$d_{D_{m(n)}}(F_nv_z,F_nv_w)\le 4^{p(n)+k}$. On the other hand,
the definition of $c_{\tau, i}$, $c_{\tau, j}$ yields
$d_{L_n}(v_z,v_w)\ge \frac124^{t}$, and thus
$d_{D_{m(n)}}(F_nv_z,F_nv_w)\ge \frac12\, 4^{p(n)+t}$. This
contradicts condition \eqref{E:t}.
\end{proof}

\begin{lemma}\label{L:Decr8} Two of the subdiamonds $S(c_{\tau,0}),S(c_{\tau,1})$,
$S(c_{\tau,2}),S(c_{\tau,3})$ have diameters at least $8$ times
smaller than the diameter of $S(c_\tau)$.
\end{lemma}

\begin{proof} Lemma \ref{L:TopBot}
excludes all subdiamonds of  $S(c_\tau)$ of diameter
$\frac12\diam(S(c_\tau))$ as candidates for
$S(c_{\tau,0}),S(c_{\tau,1}),S(c_{\tau,2}),S(c_{\tau,3})$. What is
more, Lemma \ref{L:TopBot} excludes $8$ out of $16$ subdiamonds of
$S(c_\tau)$ of diameter $\frac14\diam(S(c_\tau))$, as can be
viewed by sketching $D_3$. Consequently, we are left with two
$4$-tuples of subdiamonds of $S(c_\tau)$ of diameter
$\frac14\diam(S(c_\tau))$. Subdiamonds in each $4$-tuple have a
common vertex. Therefore, by Lemma \ref{L:DoNotTouch} only two of
these subdiamonds can serve as one of
$S(c_{\tau,0}),S(c_{\tau,1}),S(c_{\tau,2}),S(c_{\tau,3})$. The
conclusion follows.
\end{proof}

\begin{proof}[Proof of Theorem \ref{T:LaaksoIntoDmnds}] Lemma \ref{L:Decr8} implies that the diameter of at  least one of the
subdiamonds $S(c_\tau)$, where sequence $\tau$ comprises $s-t$
elements, does not exceed
\[\diam(S(c_\emptyset))\cdot
8^{-(s-t)}\le 4^{p(n)+k}\cdot \left(\frac12\, 4^s\right)\cdot
8^{-(s-t)}.\]
On the other hand, the assumption stating that $h=t$
satisfies \eqref{E:ellLarge} implies the inequality
$\diam(S(c_\tau))\ge\frac12 4^{p(n)+t-1}$ leading to
\[\frac12 4^{p(n)+t-1}\le 4^{p(n)+k}\cdot \left(\frac12\, 4^s\right)\cdot
8^{-(s-t)}\] or $2^{s-t}\le 4^{k+1}$.  This contradicts
\eqref{E:s-t}.
\end{proof}

\section{Consequences and discussion}\label{S:Conseq}

As an immediate outcome of Theorem \ref{T:LaaksoIntoDmnds} one
obtains the following statement:

\begin{corollary}\label{C:LinfNotIntoDinf} $L_\omega$ does not admit a bilipschitz
embedding into $D_\omega$.
\end{corollary}

\begin{proof} Assume the contrary. Since $L_\omega$ contains
isometric copies of (weighted) Laakso graphs
$\{L_n\}_{n=0}^\infty$, this would imply that
$\{L_n\}_{n=0}^\infty$ admits uniformly bilipschitz embeddings
into $D_\omega$. Further, $D_\omega$ is the union of
$\{D_n\}_{n=0}^\infty$ implying that $\{L_n\}_{n=0}^\infty$ admits
uniformly bilipschitz embeddings into $\{D_n\}_{n=0}^\infty$.
This, however, contradicts Theorem \ref{T:LaaksoIntoDmnds}.
\end{proof}

Corollary \ref{C:LinfNotIntoDinf} rules out the approach to the
proof of nonembeddability of $D_\omega$ into a Banach space with
the Radon-Nikod\'ym property (RNP) (\cite[Corollary 3.3]{Ost11})
by an immediate application of the corresponding result for the
Laakso space, which can be found in \cite[Corollary 1.7]{CK09} and
\cite[Theorem 3.6]{Ost11}.

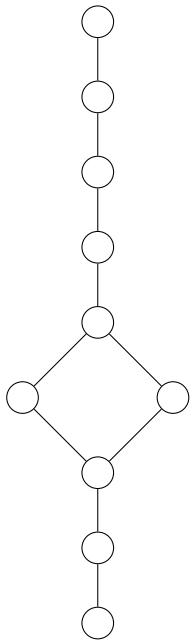
\begin{figure}
\begin{center}
{
\begin{tikzpicture}
  [scale=.25,auto=left,every node/.style={circle,draw}]
  \node (n1) at (16,0) {\hbox{~}};
  \node (n7) at (16,4)  {\hbox{~}};
  \node (n2) at (16,8)  {\hbox{~}};
  \node (n3) at (12,12) {\hbox{~}};
  \node (n4) at (20,12)  {\hbox{~}};
  \node (n5) at (16,16)  {\hbox{~}};
  \node (n6) at (16,20) {\hbox{~}};
  \node (n8) at (16,24) {\hbox{~}};
  \node (n9) at (16,28) {\hbox{~}};
  \node (n10) at (16,32) {\hbox{~}};

\foreach \from/\to in {n1/n7,n7/n2,
n2/n3,n2/n4,n4/n5,n3/n5,n5/n6,n6/n8,n8/n9,n9/n10}
    \draw (\from) -- (\to);

\end{tikzpicture}
} \caption{Laakso-type graph $M_1$.}\label{F:DmnndCompLaakso}
\end{center}
\end{figure}

However, as it can be seen from the proof above, the reason for
the non-embedda\-bi\-lity of Laakso graphs into diamond graphs is
rather specific: an incompatible choice of the parameters. One can
easily adjust Laakso graphs and get a family of Laakso-type graphs
which are isometrically embeddable into diamonds, as illustrated
by Example \ref{Ex:DmnndCompLaakso} below. An analogue of the
RNP-nonembeddability result for such Laakso-type spaces
immediately implies the non-embeddability of the infinite diamond
$D_\omega$ into Banach spaces with the RNP.

\begin{example}\label{Ex:DmnndCompLaakso} Repeat the Laakso
construction with $L_1$ replaced by the graph $M_1$ shown in
Figure \ref{F:DmnndCompLaakso}. Denote the obtained graphs
$\{M_n\}_{n=0}^\infty$. It is easy to verify that $M_1$ embeds
isometrically into $D_3$. Finally, it can be proved by induction
on $n$  that $M_n$ embeds isometrically into $D_{3n}$ for all
$n\in \mathbb{N}$.
\end{example}

Another way of proving the result simultaneously for $D_\omega$
and $L_\omega$ was found in \cite{Ost14c}. It is based on the
notion of {\it thick family of geodesics}. The result of Mendel
and Naor \cite{MN13} on the lack of Markov convexity in the Laakso
space also can be generalized  to the case of spaces with thick
families of geodesics, which includes the infinite diamond
$D_\omega$. See \cite{Ost14b} for this and related issues.

\section{Acknowledgements}

The second-named author gratefully acknowledges the support by
National Science Foundation DMS-1201269 and by Summer Support of
Research program of St. John's University during different stages
of work on this paper. The authors thank Siu Lam Leung for his
valuable comments.

\renewcommand{\refname}{\section{References}}

\end{large}

\begin{small}

\end{small}

\end{document}